\documentclass[12pt]{amsart}
\usepackage{amsfonts, amssymb, amscd}

\def\Pic 		{{\rm Pic}}
\def\vol		{{\rm vol}}

\def\ZZ                 {{\mathbb Z}}
\def\PP                {{\mathbb P}}

\def\CC                 {{\mathbb C}}

\newtheorem{lemma}{Lemma}[section]
\newtheorem{theorem}[lemma]{Theorem}
\newtheorem{corollary}[lemma]{Corollary}
\newtheorem{proposition}[lemma]{Proposition}
\theoremstyle{definition}
\newtheorem{definition}[lemma]{Definition}
\newtheorem{example}[lemma]{Example}

\newtheorem{remark}[lemma]{Remark}
\theoremstyle{remark}
\newtheorem*{proof*}{Proof}
\numberwithin{equation}{section}

\title{On complete intersections with trivial canonical class}
\author{Lev A. Borisov and Zhan Li}
\address{Mathematics Department, Rutgers University, 110 Frelinghuysen Rd, Piscataway, NJ 08540, USA}
\email{borisov@math.rutgers.edu, lizhan@math.rutgers.edu}

\begin{document}
\begin{abstract}
We prove birational boundedness results on 
complete intersections with trivial canonical class of base point free divisors in (some version of) Fano 
varieties. Our results imply in particular
that Batyrev-Borisov toric construction produces only a bounded set of  Hodge numbers 
in any given dimension, even as the codimension is allowed to grow.
\end{abstract}

\maketitle

\section{Introduction}\label{intro}
Calabi-Yau manifolds are fascinating algebraic varieties which have been actively studied for the 
last 30 years. Interest in Calabi-Yau manifolds comes from at least two sources. First, these are 
higher dimensional analogs of  the widely studied elliptic curves and K3 surfaces. Second, they have 
been prominently featured in string theory 
as possible target spaces. In string theory applications, the case of dimension three is especially
relevant and much work has been done on constructing Calabi-Yau threefolds, see for example
\cite{BatMir,BB1,BN,CD,GP,Hua,Kapustka,KS,Rodland,Tonoli,Voisin}.

\medskip
There is currently no classification of Calabi-Yau threefolds,  nor does it appear to be within reach. In fact, 
it is not known whether they form a bounded set or even if their Hodge numbers are bounded,
apart from the elliptically fibered case \cite{G}.
This paper shows that one of the most prolific constructions of Calabi-Yau threefolds can only produce
a birationally bounded set of them.

\medskip
Known constructions of Calabi-Yau varieties fall into two general categories: ``toric" and ``sporadic".  We are mostly concerned with the former.
Batyrev's construction \cite{BatMir} of Calabi-Yau threefolds as crepant resolutions of generic hypersurfaces in Gorenstein Fano toric 
fourfolds yields the vast majority of known examples.  Kreuzer and Skarke's classification 
\cite{KS} of Gorenstein toric Fano fourfolds yielded 473,800,776 potentially different families of Calabi-Yau threefolds,
of which at least 30,108 are provably different, as measured by the Hodge number invariants. 

\medskip
A generalization of Batyrev's hypersurface 
construction, known as Batyrev-Borisov construction \cite{BaBo}, considers generic complete intersections 
of $k$ hypersurfaces in Gorenstein toric Fano varieties of dimension $k+3$. While for any $k$ one is guaranteed
to have only a finite number of families, it has been long expected that for $k$ large enough these complete intersections will not produce 
essentially new varieties.  However, this has never been proved, and our paper 
establishes this (and in fact a stronger) result which we will state below after making some preliminary
definitions.

\medskip
We call a variety a bpf-big Fano (see Definition \ref{bpfbig}) if it is normal with Gorenstein canonical singularities and base point  free 
and big anticanonical divisor. Also recall that a set of varieties $\mathfrak S$ is called birationally bounded if there exists a projective morphism 
$Z\to T$ with $T$ of finite type so that for every $X\in \mathfrak S$ there exists $t\in T$ with 
fiber $Z_t$ birational to $X$.

\medskip
We can now state the main result of this paper.

\smallskip
\noindent
{\bf Theorem \ref{main}.}
Consider the set of $n$-dimensional varieties $X_0$ with Gorenstein canonical singularities and trivial canonical class which can be 
written as a connected component of a scheme-theoretic intersection of 
base point free divisors $F_i\in |D_i|$ in an $(n+k)$-dimensional bpf-big Fano variety $P$
with $\sum_{i=1}^k D_i =-K_P$. Then this set is birationally bounded for any $n$ irrespective of $k$. 

\medskip
The proof of the above theorem is based on two main ideas. First, we use difficult results of \cite{2} 
to prove that the space of bpf-big Fano varieties with Gorenstein canonical singularities with base point free 
anticanonical class is birationally
bounded in any dimension. Then we show that a connected component of a generic (weak) Calabi-Yau complete intersection of dimension $n$ 
in an $n+k$ dimensional bpf-big Fano variety can be realized as a connected component of a complete intersection in $n+l$ dimensional 
variety with $l\leq n$. This allows us to prove birational 
boundedness of (components of) generic complete intersections of this type. Finally, we 
pass from the generic case to the special one by finding a big base point free divisor of bounded volume.

\medskip
The immediate corollary of Theorem \ref{main} is the following result, which was the main motivation
behind our project.

\smallskip
\noindent
{\bf Corollary \ref{finitebabo}.} 
Batyrev-Borisov construction produces only a birationally bounded set of Calabi-Yau varieties in any given dimension. In particular, the stringy Hodge numbers of varieties obtained by Batyrev-Borisov 
constructions are bounded in any given dimension.

\medskip
The paper is organized as follows. In Section \ref{sec.term} we set up the terminology.
In Section \ref{boundsec} we use the results and techniques of 
\cite{2} to prove birational boundedness of bpf-big Fano varieties of given dimension and birational
boundedness of 
varieties with trivial canonical class and big base point free divisor of bounded volume. In Section 
\ref{secci} we establish basic results about complete intersections in bpf-big Fano varieties. 
In Section \ref{secmain} we prove the main results of the paper. Finally, we use Section \ref{seccomm}
to state some consequences and raise open questions related to this paper's results. 

\medskip
{\bf Acknowledgments.} We thank Chenyang Xu for useful comments and references. This research 
was partially supported by the NSF grant DMS-1201466.

\section{Terminology}\label{sec.term}

In this section we set up terminology which we will use in the subsequent sections. All of our varieties 
are defined over $\CC$ and are projective.

\medskip
We start by recalling birational boundedness and boundedness of  sets of varieties.
\begin{definition}\label{bound}
A set of varieties $\mathfrak S$ is called birationally bounded (resp. bounded) if there is a projective
morphism 
$Z\to T$ with $T$ of finite type so that for every $X\in \mathfrak S$ there exists $t\in T$ with 
fiber $Z_t$ birational (resp. isomorphic) to $X$.
\end{definition}

\medskip
We now define versions of Calabi-Yau and Fano conditions relevant to this paper.
\begin{definition}\label{wCY}
A complex projective variety $X$ is called  weak Calabi-Yau variety if it is normal, has trivial canonical class and Gorenstein canonical 
singularities.
\end{definition}

\begin{remark}
Note that we do not require $h^i(X,{\mathcal O}_X)=0$ for $0<i<\dim X$ and allow mild singularities.
\end{remark}

\begin{definition} \label{bpfbig}
A complex projective variety $P$  is called a bpf-big Fano variety if the following 
conditions are satisfied.
\begin{itemize}
\item $P$ is normal,
\item singularities of $P$ (if any) are Gorenstein and canonical,
\item $-K_P$ is a base point free Cartier divisor on $P$,
\item $-K_P$ is big.
\end{itemize}
\end{definition}

\begin{remark}
The bpf-big Fano condition can be compared with the more common weak Fano condition,
which assumes that $-K_P$ is nef and big. On one hand, we require $-K_P$ to be base point
free. On the other hand, 
we do not require $P$ to be smooth, contrary to the common definition of weak Fano varieties.
\end{remark}

\smallskip
The following is the prototypical example that was our main motivation in considering bpf-big Fano varieties.
\begin{example}
For a reflexive polytope $\Delta$ in a lattice $M$ the
toric variety ${\rm Proj}(\bigoplus_{l\geq 0} \CC[l\Delta\cap M])$ and any of its partial projective crepant resolutions are bpf-big Fano 
varieties. Indeed, toric singularities are normal and Cohen-Macaulay \cite{Ho}
and the reflexivity assures that the anticanonical class of ${\rm Proj}(\bigoplus_{l\geq 0} \CC[l\Delta\cap M])$ is
the class of the ample tautological line bundle $\mathcal O(1)$, see \cite{Batduke}.
It can be shown \cite[Proposition 1.2]{Mav} that this construction describes all toric bpf-big Fano varieties, so there are only a finite number of them in any given dimension.
\end{example}

\section{Bounds on bpf-big Fano varieties and weak Calabi-Yau varieties with big 
base point free divisor of bounded volume}\label{boundsec}

In this section we prove that the set of bpf-big Fano varieties is birationally bounded in any dimension. 
Our technique uses difficult results of \cite{2} on boundedness of pairs, although earlier results might be 
sufficient, since we are helped by the base point free condition. We also prove that the set 
of weak Calabi-Yau varieties of given dimension is birationally bounded under the assumption of 
existence of a big base point 
free Cartier divisor of bounded volume. As a consequence we deduce that  the set of possible stringy $E$-functions of weak Calabi-Yau varieties of fixed dimension with base point free divisor of bounded volume is finite.
These results will be instrumental in the subsequent sections.


\smallskip
The first  result of this section is the following boundedness theorem.
\begin{theorem}\label{boundFano}
For any given $n$ the set ${\mathfrak M}_n$ of $n$-dimensional bpf-big Fano varieties is birationally bounded.
In addition, $(-K_P)^n$ is bounded from above
by a constant that depends only on $n$.
\end{theorem}

\begin{proof}
We may assume that $n\geq 2$, since $n=1$ case is obvious.

\smallskip
Suppose $P\in {\mathfrak M}_n$. Then since $-2K_P$ is big and base point free,
the generic elements of $|\hskip -2pt -2K_P|$ are irreducible (\cite[Theorem I.6.3(4)]{J}). We also claim that for generic $E\in |-2K_P|$, 
the pairs
$(P,\frac 23 E)$ and $(P,\frac 12 E)$ are Kawamata log-terminal, see the argument in the next paragraph.

\smallskip
Indeed, let $\pi:\hat P\to P$ be a resolution of singularities whose exceptional divisor
$\bigcup_i F_i$ has simple normal crossings.
Denote by $V_i=\pi(F_i)\subset P$ the images of the components of the exceptional divisor. 
Base point free condition on $(-K_P)$
ensures that the generic element $E$ of $|\hskip -2pt -2K_P|$ intersects $V_i$ transversely, thus 
the strict transform $\hat E$ of $E$ coincides with the pullback $\pi^*(E)$.
Since $\pi^*|\hskip -2pt -2K_P|$ is a base point free linear system, the generic member is smooth and transversal to
$F_i$. Thus $\pi$ provides a log-resolution of $(P,\alpha E)$ for any $\alpha$ and
discrepancies can be calculated directly. The fact that $P$ has 
canonical and hence log-terminal singularities, together with $\alpha <1$ finishes the argument.

\smallskip
We now invoke Theorem B of \cite{2} which states that for 
Kawamata log-terminal pairs $(P,\Delta)$ of dimension $n$ with $K_P+\Delta$ numerically trivial and
coefficients of $\Delta$ in a subset of $[0,1)$ that satisfies DCC condition, the set 
$\vol (P,\Delta)$ is bounded from above. In our setting $K_P+\frac 12 E=0$, the set $\{\frac 12\}$ 
trivially satisfies DCC condition, and base point free condition on $-K_P$ ensures that 
$\vol(P,\frac 12 E)= (-K_P)^n$. We thus conclude that $(-K_P)^n$ is bounded from above.

\smallskip
We now use Theorem C of \cite{2}, which says that for log-canonical pairs $(P,\Delta)$
with coefficients in a DCC subset of $[0,1]$ with $K_P+\Delta$ big, there exists a constant
$m$ that depends only on $n$ and the aforementioned subset such that $|m(K_P+\Delta)|$ gives a birational map. We apply it to $(K_P,\frac 23 E)$
to show that there is $m$ that depends on $n$ only such that $|\hskip -2pt -mK_P|$ gives a birational morphism
$P\to W$. By taking a linear projection if necessary, we can assume that $W$ is a subvariety of bounded degree in a projective space of bounded dimension. Thus we can pick the appropriate union of Chow varieties
to be the base of the family $T$ in Definition \ref{bound} and the claim follows.
\end{proof}

\smallskip
We will now use similar techniques to bound weak Calabi-Yau varieties. Specifically,
for any $n\geq 0$ and $d>0$ let ${\mathfrak N}_{n,d}$ be 
the set of $n$-dimensional weak Calabi-Yau varieties $X$
that possess a base point free big Cartier divisor $D$ with $D^n\leq d$.
\begin{theorem}\label{boundCY}
For any $n$ and $d$ the set ${\mathfrak N}_{n,d}$ is birationally bounded.
\end{theorem}

\begin{proof}
Similar to the proof of Theorem \ref{boundFano}, we pick a generic member $E$ of $|D|$ and use
\cite[Theorem C]{2} to bound $m$ such that $\phi_{mD}$ is birational. Then a linear projection of the image of $\phi_{mD}$ is a variety of bounded degree $(\leq m^nd)$ in a projective space of bounded dimension $(\leq 2n+1)$. 
\end{proof}

\begin{corollary}\label{boundhodge}
For given $n$ and $d$, the set of stringy $E$-functions \cite{Batyrev-stringy} 
 $E_{st}(X;u,v)$ of $X\in {\mathfrak N}_{n,d}$ is finite.
\end{corollary}

\begin{proof}
By Theorem \ref{boundCY} and Definition \ref{bound} there exists a family
$Z\to T$ such that every member of $ {\mathfrak N}_{n,d}$ is birational to
a fiber $Z_t$ of it. We may assume the fibers to be smooth. Indeed, for any irreducible component of $T$
we can resolve the singularities of the general fiber: this will give us a resolution of fibers over a Zariski open subset of the component. We then replace $T$ by the disjoint union of these open sets and their complements, and repeat the procedure.

\smallskip
For every fiber $Z_t$ that is birational to some $X\in  {\mathfrak N}_{n,d}$ 
the dimension of $H^0(Z_t, K_{Z_t})$ is $1$, because global sections of the canonical class are a birational invariant of varieties with canonical singularities. By upper-semicontinuity, there is a Zariski locally closed 
subset of $t\in T$ where $h^0(Z_t, K_{Z_t})=1$ and we may replace $T$ by it. Now we similarly
resolve the singularities 
of the global sections of $K_{Z_t}$ to further assume that this is a simple normal crossing divisor (with integer multiplicities of the components). Then Batyrev's definition of stringy $E$-functions shows that 
$E_{st}(Z_t,K_{Z_t})$ is a constructible function and thus takes only a finite number of values (viewed as
a map $ T \to \ZZ(u,v)$). It remains to observe that $(Z_t,K_{Z_t})$ is $K$-equivalent to 
$(X,0)$ in the sense that they pull back to the same pair on a common log-resolution. Then
$$
E_{st}(X;u,v)= E_{st}(X,0;u,v) = E_{st}(Z_t,K_{Z_t};u,v)
$$
by \cite[Theorem 3.8]{Batlog}.
\end{proof}

\smallskip
\begin{remark}
It is known \cite{Kollar} that the set of true Fano varieties of fixed dimension $n$ is bounded. 
It is suggested by Xu \cite{X} that the same should be true for bpf-big Fano varieties. He has developed an approach to this statement, 
based on the arguments of \cite{BCHM}. This can presumably strengthen all of our results
from birational boundedness to boundedness in the sense of Definition \ref{bound}.
However, we choose to focus on birational boundedness results only since from the string theory point 
of view birational Calabi-Yau varieties are conjectured to correspond to different regions in the same family of field theories.
So birational boundedness and boundedness (conjecturally) translate into the same finiteness of the number of families of field theories.
\end{remark}

\section{Calabi-Yau complete intersections of base point free divisors in Fano varieties}
\label{secci}
In this section we establish basic facts about complete intersections in bpf-big Fano varieties 
that can be assured to
be weak Calabi-Yau varieties with at worst Gorenstein canonical singularities. 


\smallskip
We start by the following basic statement about complete intersections in bpf-big Fano varieties,
see Definition \ref{bpfbig}.
\begin{proposition}\label{genci}
Let $P$ be a bpf-big Fano variety of dimension $n+k$ and let 
$$ D_1+\cdots+ D_k = -K_P$$
be a decomposition of $-K_P$ into a sum of base point free Cartier divisor classes $D_i$ in $\Pic(P)$.
Then for a generic choice of sections $F_i\in |D_i|$, the scheme-theoretic intersection
$$
X=F_1\cap \cdots \cap F_k
$$
is a disjoint union of a finite set of weak Calabi-Yau varieties of dimension $n$ 
in the sense of Definition \ref{wCY}.
Moreover, $-K_P\vert_{X_0}$ is a nef and big line bundle
on every connected component $X_0$ of $X$.
\end{proposition}

\begin{proof}
Let $\pi:\hat P \to P$ be a resolution of singularities of $P$ with exceptional divisor $E$ with simple normal crossings. 
The pullbacks of $\pi^* D_i$ to $\hat P$ are base point free divisors which add up to $\pi^*(-K_P)$
and pullbacks of $|D_i|$ give base point free linear systems on $\hat P$. By repeated application of Bertini's theorem, we see that 
for generic $F_i$ the complete intersection of the corresponding sections 
$\pi^*(F_i)$ of $\pi^*(D_i)$ is a smooth $n$-dimensional variety $\hat X$. Moreover, by using Bertini's
theorem for the restrictions of $\pi^*|D_i|$ to the (intersections of) components of $E$ we see that $\pi^*X \cap E$ is  a simple 
normal crossing divisor on $\hat X$.

\smallskip
Thus, $\hat X\to X$ is a resolution of singularities of $X$. The adjunction formula then implies that 
the canonical class of $X$ is $0$. Finally, discrepancies of the components of $X\cap E$ coincide
with discrepancies of components of $E$, which are nonnegative by Definition \ref{bpfbig}.

\smallskip
It remains to show that $-K_P\vert_{X_0}$ is nef and big on $X_0$. It is automatically nef as a pullback of a nef 
divisor. Consider the exceptional locus $G$ of the map 
\begin{equation}\label{mu}
\mu:P\to {\rm Proj}(\bigoplus_{l\geq 0 } H^0(P,-lK_P))
\end{equation}
where we use base point free property of $-K_P$ to show that the graded ring above is finitely generated.
The bpf-big Fano assumption also implies that this map is birational.
We need to show that $X_0$ is not contained in $G$. However, by base point free property of $D_i$ we can assure
that $\dim( X \cap G_j)\leq \dim G_j-k$ for every irreducible component of $G_j$. Thus, $\dim (X_0\cap G_j)<n$ and the 
statement follows.
\end{proof}

\smallskip
\begin{remark}
One can not assume connectedness of complete intersections. For example, complete intersection
on $\CC\PP^4\times \CC \PP^1$ of the $(5,0)$ and $(0,2)$ divisors is a disjoint union of two copies of a quintic
Calabi-Yau threefold. Similarly, it is difficult to assure the vanishing of middle cohomology of $\mathcal O(X)$,
as example of $(3,0)$ and $(0,3)$ divisors in $\CC\PP^2\times \CC\PP^2$ shows.
\end{remark}

\smallskip
\begin{definition}\label{defgen}
We say that a Calabi-Yau complete intersection $X$ in $P$ is generic if it fits the description above 
for some resolution of singularities $\hat P \to P$, and moreover for any subset $I\subseteq \{1,\ldots, k\}$ 
the complete intersection of $\pi^* F_i,i\not\in I$ is smooth on $\hat P$ and is transversal to $E$.
\end{definition}

\section{The main result}\label{secmain}
In this section we prove our main result, namely that weak Calabi-Yau varieties
of dimension $n$
that are obtained as connected components of  complete intersections
of base point free divisors in bpf-big Fano varieties are birationally bounded. The key idea is to reduce the boundedness of complete intersections of dimension $n$ and arbitrary codimension to 
those of codimension at most $n$.

\smallskip
\begin{theorem}\label{gentheorem} 
Let $n$ be a fixed integer. Consider the set of $(n+k)$-dimensional bpf-big
Fano varieties $P$ with decompositions
$$
-K_P=\sum_{i=1}^k D_i
$$
where $D_i$ are base point free Cartier divisors. Then the set of connected components of  
generic complete intersections $X$ of sections of $D_i$ in the sense of Definition \ref{defgen}
is birationally bounded.
\end{theorem}

\begin{proof}
Because the divisors $D_i$ are Cartier, the intersection numbers $(\prod_i {D_i}^{s_i})Y$ for $\sum_i s_i=l$
are well-defined integers for any $l$-cycle $Y\in A_l(P)$. Moreover, the base point free assumption on $D_i$ implies that $(\prod_i {D_i}^{s_i})Y \geq 0$ for an effective $Y$.
The last sentence of Proposition \ref{genci}  
implies that 
$$
(\sum_{i=1}^k D_i)^{n}X_0 > 0,
$$
thus there exists a subset $I\subseteq \{1,\ldots, k\}$ of cardinality at most $n$ and positive integers $s_i,i\in I$ with $\sum_{i\in I}s_i=n$ such that
$$(\prod_{i\in I} D_i^{s_i}) X_0 >0.$$

\smallskip
Consider the connected component $P_0\supseteq X_0$ of the scheme-theoretic complete intersection 
$$P_1= \bigcap_{i\not\in I}  F_i$$
in $P$.  We note that $X_0$ is a connected component of the complete intersection of $D_i,i\in I$
on $P_0$. By Definition \ref{defgen} and arguments similar to Proposition \ref{genci}, we see that 
$P_0$ has at worst Gorenstein canonical singularities and $-K_{P_0} = \sum_{i\in I} D_i\vert_{P_0}$
is nef. Moreover,
\begin{equation}\label{boundvol}
(\sum_{i\in I} D_i)^{n+|I|}P_0\geq 
(\sum_{i\in I} D_i)^{n} (\prod_{i\in I} D_i) P_0\\
\geq (\sum_{i\in I}D_i)^nX_0
 \geq
(\prod_{i\in I} D_i^{s_i} )X_0 >0
\end{equation}
which implies that $P_0$ is a bpf-big Fano variety of dimension $n+|I|\in [n+1,2n]$.

\smallskip
Theorem \ref{boundFano} now  implies that 
$(-K_{P_0})^{n+|I|}$ is bounded by a constant that depends only on $n$. 
The divisor $-K_{P_0}\vert_{X_0} = \sum_{i\in I} D_i\vert_{X_0}$ is base point free.
Moreover, the inequalities  \eqref{boundvol} imply that 
$$
(-K_{P_0})^{n+|I|}\geq (\sum_{i\in I} D_i\vert_{X_0})^n >0
$$
so $-K_{P_0}\vert_{X_0} $ is big and of bounded volume. 
Theorem \ref{boundCY} now finishes the proof.
\end{proof}

\smallskip
While it is satisfying to know that generic weak Calabi-Yau complete intersections of given dimension
and arbitrary codimension are birationally bounded, we want to also exclude the possibility of getting birationally unbounded sets
of weak Calabi-Yau varieties by non-generic complete intersections in some Fano 
varieties. Of course, for an individual $P$ these non-generic complete intersections are degenerations of 
the generic ones, but it is a priori conceivable that one can get more and more complicated degenerations
as one considers subfamilies of a given family (because it is not clear if we have any separatedness of the
moduli of Calabi-Yau varieties). The following theorem establishes this stronger result.

\begin{theorem}\label{main}
Consider the set of $n$-dimensional varieties $X_0$ with Gorenstein canonical singularities and trivial canonical class which can be 
written as a connected component of a scheme-theoretic intersection of 
base point free divisors $F_i\in |D_i|$ in an $(n+k)$-dimensional bpf-big Fano variety $P$
with $\sum_{i=1}^k D_i =-K_P$. Then this set is birationally bounded for any $n$ irrespective of $k$. 
\end{theorem}

\begin{proof}
The main idea of the proof is that because $X_0$ is reduced, once one starts to deform 
$F_i$ in their linear system, $X_0$ will
deform to a single  connected component of a generic complete intersection as opposed to a union of several 
connected components. We will be more specific below.

\smallskip
Pick sections $\tilde F_i$ of $D_i$ so that for general $\epsilon$ the complete intersection
of $F_i+\epsilon \tilde F_i$ is generic in the sense of Definition \ref{defgen}.
Consider the scheme theoretic intersection 
$$Y=\bigcap_{i=1}^k\{uF_i+v\tilde F_i=0\}$$
in $P\times \PP^1$ where $\PP^1$ has 
homogenous coordinates $(u:v)$.
Let $Y_0$ be the irreducible component of $Y$ that contains $X_0\times \{(1:0)\}$. 
Note that $Y_0\to \PP^1$ is flat and  $X_0$ is the fiber of it over $(1:0)$.
Since $h^0(X_0,{\mathcal O}_{X_0})=1$, by flatness and upper semicontinuity 
we see that for a generic fiber $X_\epsilon$ there holds $h^0(X_\epsilon,{\mathcal O}_{X_\epsilon})=1$.
Thus for a general $\epsilon$ the fiber $X_\epsilon$ consists of one connected component 
of complete intersection rather  than several. 
We also see that $X_0$ and $X_\epsilon$ are rationally equivalent $n$-cycles in $P$.

\smallskip
By the argument of the proof of Theorem \ref{gentheorem} there exists a set $I$ of cardinality at most $n$ such that 
$\sum_{i\in I}D_i$ resrticts to a nef and big divisor on $X_\epsilon$. Since $X_0$ and $X_\epsilon$ 
are rationally equivalent cycles in $P$, we have  
$(\sum_{i\in I}D_i)^nX_0=(\sum_{i\in I}D_i)^nX_\epsilon$. The latter is positive, but bounded by a constant that depends only on $n$, see \eqref{boundvol} in
the proof of Theorem \ref{gentheorem}. Thus $\sum_{i\in I} D_i\vert_{X_0}$ is big (and is base point free because 
$D_i$ are) with bounded volume, so 
Theorem \ref{boundCY} finishes the proof.
\end{proof}

\smallskip
An immediate corollary of the above theorems is the following.
\begin{corollary}\label{finitebabo}
Batyrev-Borisov construction produces only a birationally bounded set of Calabi-Yau varieties in any given dimension. In particular, the stringy Hodge numbers of varieties obtained by Batyrev-Borisov 
constructions are bounded in any given dimension.
\end{corollary}

\begin{proof}
Batyrev-Borisov construction is a particular case of the complete intersection construction 
when the ambient bpf-big Fano variety $P$ is toric. The argument of Corollary \ref{boundhodge}
then allows one to conclude that stringy $E$-functions and thus Hodge numbers of these varieties 
are bounded.
\end{proof}

\begin{remark}
While stringy $E$-functions in Batyrev-Borisov construction are polynomial \cite{BaBo}, they are not expected to
be so in general, see examples of \cite{DR}.
\end{remark}

\section{Comments and questions.}\label{seccomm}
In this section we make several comments and highlight some natural open questions.

\medskip
A complete intersection construction of Calabi-Yau varieties $X$ may be generalized to sections 
of globally generated vector bundles of rank $k$  and $c_1=-K_P$ on $n+k$ dimensional Fano varieties $P$.
Our methods do not appear to yield (birational) boundedness of such $X$, although it can be reasonably expected.

\medskip
One can not expect boundedness of families of complete intersections with trivial canonical 
class in varieties $P$ with $-K_P$ nef but not big. For example, $\PP^2\times S$ for $K3$ surfaces 
$S$ are not birationally bounded and neither are the anticanonical hypersurfaces in them.

\medskip
Our results imply that there is a bound $N$ such that a $K3$ surface with Picard group of rank one 
and generator $D$ with $D^2>N$ is never a base point free complete intersection in a bpf-big Fano variety.
Indeed, in dimension two birational equivalence coincides with $K$-equivalence, so 
smooth birational $K3$ surfaces are isomorphic.
We do not venture an estimate on such $N$, since the methods of this paper are not suitable for providing
a reasonable effective bound. Certainly, if a generic polarized K3 surface
with a given $D^2=2l$ could be written as a complete intersection, then the moduli space of such surfaces
would be unirational, which is known to fail for large enough $l$, see \cite{GHuS}. 
A similar statement can be made for polarized abelian surfaces.

\medskip
Our result for complete intersections in toric varieties is obtained by a decidedly non-toric method.
It would be interesting to see a more combinatorial proof of it. In particular, our main idea is that 
a generic dimension $n$ complete intersection of arbitrary codimension in a bpf-big Fano variety
can be realized as 
a complete intersection of codimension at most $n$. Is there a similar but perhaps 
larger bound to assure that a generic complete intersection in a toric Fano variety can be reduced 
to a complete intersection in toric Fano variety of bounded codimension?

\medskip
Batyrev-Borisov construction provides analogs of Hodge numbers for pairs of reflexive Gorenstein cones.
In the case when one can associate Calabi-Yau complete intersections to these cones, these Hodge numbers
are the stringy Hodge numbers of these intersections and are thus bounded. It is at the moment an open 
question whether the boundedness extents to the reflexive Gorenstein cone construction. This is important, because conjecturally  this construction provides more general examples of conformal field theories.

\end{document}